\documentclass[11pt,a4paper,twoside,leqno]{article}
\pdfoutput=1
\usepackage{comment}

\usepackage{amsmath}
\usepackage{amsthm}
\usepackage{thmtools}
\usepackage{braket}
\usepackage{mathtools}

\usepackage{enumitem}
\setlist[enumerate]{label=\textnormal{(\roman*)}}

\usepackage{fullpage}
\usepackage[dvipsnames]{xcolor}
\definecolor{britishracinggreen}{rgb}{0.0, 0.26, 0.15}
\definecolor{cobalt}{rgb}{0.0, 0.28, 0.67}

\usepackage[english]{babel}
\usepackage[utf8]{inputenc}
\usepackage[T1]{fontenc}
\usepackage{libertine}
\usepackage[libertine]{newtxmath}
\usepackage[scaled=0.83]{beramono}
\usepackage{mathrsfs}
\usepackage{soul}
\usepackage{microtype}
\DeclareSymbolFont{usualmathcal}{OMS}{cmsy}{m}{n}
\DeclareSymbolFontAlphabet{\mathcal}{usualmathcal}
\DeclareMathAlphabet\BCal{OMS}{cmsy}{b}{n}
\usepackage{hyperref}
\hypersetup{
  hypertexnames=false,
  colorlinks,
  citecolor=britishracinggreen,
  filecolor=black,
  linkcolor=cobalt,%
  urlcolor=black
}
\usepackage[capitalize]{cleveref}

\numberwithin{equation}{section}

\newcommand\dash{\nobreakdash-\hspace{0pt}}

\def\into{\hookrightarrow}
\def\onto{\twoheadrightarrow}

\def\Sets{\mathrm{Sets}}

\def\LL{\mathbf{L}}

\newcommand\Ocal{\ensuremath{\mathcal{O}}}
\newcommand\Xcal{\ensuremath{\mathcal{X}}}
\newcommand\Ycal{\ensuremath{\mathcal{Y}}}

\newcommand\fTd{\ensuremath{\mathrm{fTd}}}

\newcommand\bounded{\ensuremath{\mathrm{b}}}

\DeclareMathOperator\derived{\mathbf{D}}
\DeclareMathOperator\Inj{Inj}
\DeclareMathOperator{\QCoh}{Qcoh}
\DeclareMathOperator\Qcoh\QCoh %
\DeclareMathOperator\coh{coh}
\DeclareMathOperator\KKK{\ensuremath{\mathbf{K}}}
\DeclareMathOperator{\Perf}{Perf}

\DeclareMathOperator\identity{id}

\DeclareMathOperator{\Spf}{Spf}
\DeclareMathOperator{\Aff}{Aff}

\DeclareMathOperator{\op}{op}
\DeclareMathOperator\obstr{\mathfrak{o}}

\DeclareMathOperator{\Spec}{Spec}
\DeclareMathOperator{\Ind}{Ind}
\DeclareMathOperator{\Pro}{Pro}
\DeclareMathOperator{\Proj}{Proj}

\DeclareMathOperator{\Ext}{Ext}

\DeclareMathOperator{\Hom}{Hom}

\DeclareMathOperator{\RRHom}{\mathbf{R}Hom}
\DeclareMathOperator{\RRlHom}{\mathbf{R}\kern-0.025em\mathscr{H}\kern-0.3em\textit{om}}
\DeclareMathOperator{\lHom}{\mathscr{H}\kern-0.3em\textit{om}}

\DeclareMathOperator{\End}{End}

\DeclareMathOperator{\pr}{pr}

\DeclareMathOperator\HH{H}

\renewcommand{\geq}{\geqslant}

\declaretheoremstyle[spaceabove = 3pt, spacebelow = 3pt, bodyfont = \itshape]{theorem}
\declaretheoremstyle[spaceabove = 3pt, spacebelow = 3pt]{remark}

\declaretheorem[style=theorem, numberwithin=section]{theorem}

\declaretheorem[style=theorem, sibling=theorem]{lemma}

\declaretheorem[style=theorem, sibling=theorem]{proposition}

\declaretheorem[style=remark, sibling=theorem]{definition}
\declaretheorem[style=remark, sibling=theorem]{example}
\declaretheorem[style=remark, sibling=theorem]{remark}

\declaretheorem[style=theorem, numberwithin=section, title=Theorem]{alphatheorem}

\declaretheorem[style=theorem, sibling=alphatheorem, title=Corollary]{alphacorollary}

\crefname{alphatheorem}{Theorem}{Theorems}
\crefname{alphaconjecture}{Conjecture}{Conjectures}
\crefname{alphacorollary}{Corollary}{Corollaries}
\crefname{alphaproposition}{Proposition}{Propositions}

\usepackage{tikz}
\usepackage{tikz-cd}
\usepackage{rotating}
\newcommand*{\isoarrow}[1]{\arrow[#1,"\rotatebox{90}{\(\sim\)}"
]}
\usetikzlibrary{matrix,shapes,arrows,decorations.pathmorphing}

\tikzset{commutative diagrams/.cd,
mysymbol/.style={start anchor=center,end anchor=center,draw=none}}
\newcommand\MySymb[2][\square]{%
\arrow[mysymbol]{#2}[description]{#1}}
\tikzset{
  shift up/.style={
    to path={([yshift=#1]\tikztostart.east) -- ([yshift=#1]\tikztotarget.west) \tikztonodes}
  }
}

\DeclareMathAlphabet{\mathpzc}{OT1}{pzc}{m}{it}

\newcommand*{\defeq}{\mathrel{\vcenter{\baselineskip0.5ex \lineskiplimit0pt
  \hbox{\scriptsize.}\hbox{\scriptsize.}}}%
=}

\DeclareMathOperator\lift{L}

\usepackage[
  final
]{showlabels}

\makeatletter
\def\SL@eqntext#1{\rlap{\quad{\showlabelsetlabel{\SL@prlabelname{#1}}}}}
\makeatother

\DeclareMathOperator{\dec}{\mathsf{DEC}}

\newcommand{\id}{\operatorname{id}}
\newcommand{\simto}{\,\widetilde{\to}\,}

\newcommand{\cA}{\ensuremath{\mathcal{A}}}
\newcommand{\cB}{\ensuremath{\mathcal{B}}}
\newcommand{\cC}{\ensuremath{\mathcal{C}}}

\newcommand{\cE}{\ensuremath{\mathcal{E}}}

\newcommand{\cO}{\ensuremath{\mathcal{O}}}

\newcommand{\cX}{\ensuremath{\mathcal{X}}}
\newcommand{\cY}{\ensuremath{\mathcal{Y}}}

\newcommand{\bZ}{\ensuremath{\mathbb{Z}}}

\newcommand{\Ibar}{\overline{I}}

\newcommand{\bfk}{\mathbf{k}}

\usepackage{todonotes}
\setuptodonotes{inline}

\definecolor{caribbeangreen}{rgb}{0.0, 0.8, 0.6}

\definecolor{brightpink}{rgb}{1.0, 0.0, 0.5}

\begin{document}

\title{Deformation theory for a morphism in the derived category with fixed lift of the codomain}
\author{Pieter Belmans \and Wendy Lowen \and Shinnosuke Okawa \and Andrea T.\ Ricolfi}

\maketitle

\begin{abstract}
  We develop the deformation-obstruction calculus
  for morphisms of complexes
  with a fixed lift of the codomain,
  to derived categories of flat nilpotent deformations of abelian categories.
  As an application, we give an alternative proof
  that semiorthogonal decompositions deform uniquely
  in smooth proper families of schemes.
\end{abstract}

{
  \hypersetup{linkcolor=black}
  \setcounter{tocdepth}{3}
  \tableofcontents
}

\section{Introduction}
\label{section:introduction}
The deformation theory of abelian categories,
and a deformation-obstruction calculus describing it,
was introduced in \cite{MR2183254,MR2238922}.
Subsequently,
the deformation theory for lifting \emph{objects} to
(derived categories of) flat nilpotent deformations of abelian categories
was developed in \cite{MR2175388}.
In this paper we work out the next step:
the deformation-obstruction calculus
for lifting \emph{morphisms} of complexes
with a fixed lift of the codomain,
to derived categories of flat nilpotent deformations of abelian categories.

As an application of this novel deformation theory,
we give an alternative proof of
the uniqueness of deformations of semiorthogonal decompositions,
which is a key ingredient in \cite{2002.03303}.

\paragraph{Setup and notation}
Because it is a generalisation of the deformation-obstruction theory from \cite{MR2175388},
we will use the notation and terminology from op.~cit.
We fix a sequence of surjections of noetherian (or coherent) rings
\begin{equation}
  \label{equation:rings}
  \overline{R}\twoheadrightarrow R\twoheadrightarrow R_0
\end{equation}
where we denote\footnote{These are $I$ and $J$ in \cite{MR2175388};
  we use different symbols, to avoid confusion with our notation for injective objects.
}~$\ker(\overline{R}\twoheadrightarrow R_0)=\mathfrak{a}$
and~$\ker(\overline{R}\to R)= \mathfrak{b}$.
We moreover assume that~$\overline{R}\twoheadrightarrow R$
is a \emph{small extension} relative to~$R\twoheadrightarrow R_0$,
i.e.,~that~$\mathfrak{a}\mathfrak{b}=0$.
This implies that~$\mathfrak{b}^2=0$,
and that~$\mathfrak{b}$ has the structure of an~$R_0$\dash module.

We consider an~$R_0$\dash linear abelian category~$\mathcal{C}_0$,
which is assumed to be flat in the sense of \cite[Definition~3.2]{MR2238922}.
Depending on the situation, we will often moreover assume it is either Grothendieck or co-Grothendieck.
Then we consider a flat abelian deformation~$\mathcal{C}$ of~$\mathcal{C}_0$ over~$R$,
and a further flat abelian deformation~$\overline{\mathcal{C}}$ over~$\overline{R}$,
in the sense of \cite[\S5]{MR2238922}.

Each of the functors~$\mathcal{C}_0\to\mathcal{C}$ and~$\mathcal{C}\to\overline{\mathcal{C}}$
has two associated \emph{(left and right) restriction functors}:
its left adjoint given by the tensor product functor and its right adjoint given by the Hom functor.
We can summarise the situation in the following diagram
\begin{equation}
  \label{equation:deformations}
  \begin{tikzcd}[row sep=large]
    \overline{R}\colon & & \overline{\mathcal{C}} \arrow[d, bend right,
    "R\otimes_{\overline{R}}-"']\arrow[d, bend left, "{\Hom_{\overline{R}}(R,-)}"] \\
    R\colon &  & \mathcal{C} \arrow[d, bend right, "R_0\otimes_R-"'] \arrow[d, bend left,
    "{\Hom_R(R_0,-)}"] \arrow[u] \\
    R_0\colon & & \mathcal{C}_0 \arrow[u]
  \end{tikzcd}
\end{equation}

The following prototypical example will be the setting for our geometric applications.
\begin{example}
  \label{example:geometric-setting}
  Consider rings as in \eqref{equation:rings}.
  Let~$\overline{X} \to \Spec \overline{R}$ be
  a flat morphism of quasicompact and semiseparated schemes,
  and~$X \to \Spec R$, resp.~$X_0 \to \Spec R_0$ be its base changes.
  Then
  \begin{equation}
    \cC_0 \defeq \Qcoh X_0
    \hookrightarrow
    \cC \defeq \Qcoh X
    \hookrightarrow
    \overline{\cC} \defeq \Qcoh \overline{X}
  \end{equation}
  are flat Grothendieck abelian categories linear over~$R_0,R,\overline{R}$, respectively,
  where the inclusion functors are given by pushforwards of quasicoherent sheaves.
  The \emph{left} restriction functors~$R\otimes_{\overline{R}}-$ and~$R_0\otimes_R-$,
  in this context are simply the pullbacks of quasicoherent sheaves along the inclusions.
\end{example}

Notice that the right restriction functors preserve injectives,
while the left restriction functors preserve projectives.
We are interested in lifting objects (or complexes) and morphisms along these restriction functors,
and following \cite{MR2175388} our approach will be to replace the complexes of our interest
by appropriate complexes of injective or projective objects.
If the abelian categories~$\mathcal{A} = \mathcal{C}_0, \mathcal{C}, \overline{\mathcal{C}}$
under consideration do not have enough injectives (resp.~projectives),
one can consider their associated Ind-completions $\Ind(\mathcal{A})$
(resp.~Pro-completions $\Pro(\mathcal{A}) = \Ind(\mathcal{A}^{\op})^{\op}$).
Assuming~$\mathcal{A}$ to be small (as we may do so by tacitly enlarging the universe),
these completions are Grothendieck (resp. co-Grothendieck) categories
and hence have enough injectives (resp.~projectives).
When working with Grothendieck categories from the start,
this intermediate step is unnecessary for lifting along the right restriction functors.
However, as \cref{example:geometric-setting} shows,
in algebraic geometry it is often more natural to lift along the left restriction functors
and in this case one does need the intermediate step of considering Pro-completions.

\paragraph{Deforming morphisms in the homotopy category of injectives}
Similar to the approach in \cite{MR2175388}
we will first prove results for homotopy categories of linear categories
and their deformations
(in \cref{section:KInj-deformations}),
and then restrict to appropriate subcategories which decribe the derived categories of our interest
(in \cref{section:restriction}).
To do so, we will appeal to the comparison machinery of \cite[\S6]{MR2175388}
which explains how lift groupoids for subcategories are related to
lift groupoids for the ambient categories.

Our first theorem is the following deformation-obstruction result,
which will be proven in \cref{section:KInj-deformations}.
It is phrased in the setting of Grothendieck categories and right restriction functors,
which is also the setting used in op.~cit.~(but is dual to the setting we are eventually interested in).

In what follows, the notation~$\KKK(-)$
refers to the homotopy category of complexes of an additive category (when applied to a category)
or the functor induced between the homotopy categories (when applied to a functor).

\begin{alphatheorem}
  \label{theorem:right-adjoint-restriction}
  Let~$\mathcal{C}_0$ be a Grothendieck abelian category.
  Consider an exact triangle
  \begin{equation}
    F_0\overset{s_0}{\to}G_0\to H_0\to F_0[1]
  \end{equation}
  in~$\KKK(\Inj\mathcal{C}_0)$, and let
  \begin{equation}
    F\overset{s}{\to}G
  \end{equation}
  be a morphism in~$\KKK(\Inj\mathcal{C})$ which restricts to~$s_0$, i.e.~$\KKK(\Hom_R(R_0,s))=s_0$.

  Let~$\overline{G}$ be a lift of~$G$ from~$\KKK(\Inj\mathcal{C})$ to~$\KKK(\Inj\overline{\mathcal{C}})$
  along~$\KKK(\Hom_{\overline{R}}(R,-))$.
  \begin{enumerate}
    \item
      \label{item:right-adjoint-obstruction}
      There exists an obstruction class
      \begin{equation}
        \obstr(s)\in \HH^1(\RRHom_{\mathcal{C}_0}(\RRHom_{R_0}( \mathfrak{b},F_0),H_0))
      \end{equation}
      such that~$\obstr(s)=0$ if and only if~$s$ lifts to a
      morphism~$\overline{s}\colon\overline{F}\to\overline{G}$ in~$\KKK(\Inj\overline{\mathcal{C}})$.
    \item
      \label{item:right-adjoint-torsor}
      Assume that~$\HH^{-1}(\RRHom_{\mathcal{C}_0}(\RRHom_{R_0}( \mathfrak{b},F_0),H_0))=0$.
      If~$\obstr(s)=0$,
      then the set of isomorphism classes of such lifts is a torsor under
      \begin{equation}
        \HH^0(\RRHom_{\mathcal{C}_0}(\RRHom_{R_0}( \mathfrak{b},F_0),H_0)).
      \end{equation}
  \end{enumerate}
\end{alphatheorem}

The~$\HH^{-1}$\dash vanishing is a shadow of a derived deformation theory involving higher structures,
and is also present in \cite{MR2175388}.
We refrain from developing it here,
because we do not need it for our application (and likely our non-homotopical tools are not the best to deal with it).

\begin{remark}
  \label{remark:recover-deformation-theory-complexes}
  \cref{theorem:right-adjoint-restriction} recovers
  the deformation theory of objects from \cite[Theorem~5.7]{MR2175388}
  (in the version for the homotopy category of injectives), by considering the triangle
  \begin{equation}
    C_0\to 0\to C_0[1]\xrightarrow[\id_{C_0[1]}]{+1}C_0[1]
  \end{equation}
  and where the fixed lift of the zero object is the zero object.
\end{remark}

For our purposes we will need the result which is dual to \cref{theorem:right-adjoint-restriction},
concerning lifts along the \emph{left} restriction functor (given by the tensor product).
This is stated in \cref{theorem:left-adjoint-restriction}.

\paragraph{Deforming morphisms in the derived category}
\Cref{theorem:right-adjoint-restriction,theorem:left-adjoint-restriction}
are already useful results in an abstract and general setting.
However,
as in \cref{example:geometric-setting} we are interested
in using the left restriction functor,
yet,
categories of quasicoherent sheaves
are not co-Grothendieck.
One can take the Pro-completion of the category of quasicoherent sheaves,
making it co-Grothendieck,
but then one needs to understand the compatibility between
the deformation theory in the original derived category
and the derived category of this Pro-completion.
This is worked out for objects
in \cite[\S6]{MR2175388}.
In \cref{section:restriction}
we work out the version for morphisms.

Our main goal is the following corollary
of \cref{theorem:left-adjoint-restriction},
which is of clear algebro-geometric interest.

\begin{alphacorollary}
  \label{corollary:lifting-morphisms-of-perfect-complexes}
  Consider the situation of \cref{example:geometric-setting}.
  Consider an exact triangle
  \begin{equation}
    F_0\overset{s_0}{\to}G_0\to H_0\to F_0[1]
  \end{equation}
  in~$\Perf X_0$, and let
  \begin{equation}
    F\overset{s}{\to}G
  \end{equation}
  be a morphism in~$\Perf X$ which restricts to~$s_0$, i.e.~$R_0\otimes_R^\LL s=s_0$.

  Let~$\overline{G}$ be a lift of~$G$ from~$\Perf X$ to~$\Perf\overline{X}$
  along~$R\otimes_{\overline{R}}^\LL-$.
  \begin{enumerate}
    \item
      \label{item:perfect-obstruction}
      There exists an obstruction class
      \begin{equation}
        \obstr(s)\in\Ext_{X_0}^1(F_0, \mathfrak{b}\otimes_{R_0}^\LL H_0)
      \end{equation}
      such that~$\obstr(s)=0$ if and only if~$s$ lifts to a
      morphism~$\overline{s}\colon\overline{F}\to\overline{G}$ in~$\Perf\overline{X}$.
    \item
      \label{item:perfect-torsor}
      Assume that~$\Ext^{-1}_{X_0}(F_0, \mathfrak{b}\otimes_{R_0}^\LL H_0)=0$.
      If~$\obstr(s)=0$,
      then the set of isomorphism classes of such lifts is a torsor under
      \begin{equation}
        \Ext^0_{X_0}(F_0, \mathfrak{b}\otimes_{R_0}^\LL H_0).
      \end{equation}
  \end{enumerate}
\end{alphacorollary}

We do not exhaustively develop our deformation-obstruction calculus
in all other variations discussed in \cite[\S6]{MR2175388},
and leave this to the interested reader.

\paragraph{Deforming semiorthogonal decompositions}
We now arrive at our initial motivation for developing the deformation theory in this paper:
to show that
(linear) semiorthogonal decompositions of categories of perfect complexes
deform uniquely in smooth proper families of schemes.
For context and background on semiorthogonal decompositions,
one is referred to the companion paper \cite{2002.03303},
and Kuznetsov's ICM surveys \cite{MR3728631,MR4680279}.

The following theorem corresponds exactly to \cite[Theorem~7.7]{2002.03303},
except that it is written in terms of a different functor,
introduced in \cref{definition:dec-functor}.
The equivalence to \cite[Theorem~7.7]{2002.03303}
follows from the comparison of functors in \cite[Theorem~5.7]{2002.03303}.

\begin{alphatheorem}
  \label{theorem:uniqueness-deformation-decomposition-triangles}
  Let~$f\colon\cX\to U$
  be a smooth and proper morphism of quasicompact and semiseparated schemes.
  Let $(R,\mathfrak m,\bfk)$ be a local noetherian ring that is $\mathfrak m$-adically complete.
  For every morphism $\Spec R \to U$, the natural map
  \begin{equation}
    \begin{tikzcd}
      \vartheta\colon \dec_{\Delta_f}(\Spec R \to U) \arrow{r} & \dec_{\Delta_f}(\Spec \bfk \to U)
    \end{tikzcd}
  \end{equation}
  is bijective.
\end{alphatheorem}

The proof of \cref{theorem:uniqueness-deformation-decomposition-triangles}
in \cref{section:unique-deformation-decomposition-triangles}
corresponds to the argument in the first preprint version of \cite{2002.03303}.
However, the referee of op.~cit.~suggested a different argument
which did not rely on developing a novel deformation theory from scratch.

\paragraph{Acknowledgements}
We want to thank Michel Van den Bergh for interesting discussions.

The first author was partially supported by FWO (Research Foundation---Flanders)
during the initial phase of this project,
and during the final phase of the project
by NWO (Dutch Research Council)
as part of the grant \href{https://doi.org/10.61686/RZKLF82806}{doi:10.61686/RZKLF82806}.
The second author was supported by
the ERC (European Research Council)
under the European Union's Horizon 2020 research and innovation programme (grant agreement no.~817762),
and FWO (Research Foundation--Flanders) (grant no.~G076725N).
The third author was partially supported by Grants-in-Aid for Scientific Research
(16H05994,
  16K13746,
  16H02141,
  16K13743,
  16K13755,
  16H06337,
  19KK0348,
  20H01797,
  20H01794,
  21H04994,
23H01074)
and the Inamori Foundation.
Finally, we want to thank the Max Planck Institute for Mathematics (Bonn)
and SISSA for the pleasant working conditions.

\section{Deformation theory for morphisms in the homotopy category of injectives}
\label{section:KInj-deformations}
In this section we prove \cref{theorem:right-adjoint-restriction}.
Let the notations be as in the theorem.
We need a convenient way of representing the morphism~$s\colon F\to G$.
We do this by using the coderived model structure
on the category of cochain complexes~$\operatorname{Ch}(\mathcal{C})$.

Let us briefly recall the relevant facts.
For a Grothendieck category $\mathcal{C}$,
the category~$\operatorname{Ch}(\mathcal{C})$ can be endowed with a (cofibrantly generated) abelian model structure,
for which $\KKK(\Inj\mathcal{C})$ is the homotopy category,
see \cite[\S7]{MR4781731}.
In this model category structure,
all objects are cofibrant,
the fibrant objects are the graded-injectives complexes
(in other words, the complexes in $\operatorname{Ch}(\Inj\mathcal{C})$),
and the weakly trivial objects (or coacyclic objects)
are the objects left orthogonal to the fibrant ones.
As is the case in an abelian model category,
these classes of objects determine the classes of cofibrations, fibrations and weak equivalences.
In particular, the cofibrations are the degreewise monomorphisms.

\begin{lemma}
  \label{lemma:presentation}
  Let~$\mathcal{C}$ be a Grothendieck category,
  and~$s\colon F\to G$ a morphism in~$\operatorname{Ch}(\Inj\mathcal{C})$.
  There exists a homotopy equivalence $h\colon G \to G'$
  with $G'\in \operatorname{Ch}(\Inj\mathcal{C})$
  such that the composition $h\circ s$ is homotopic to a degreewise monomorphism.
\end{lemma}
\vspace{-.5\baselineskip}
\begin{proof}
  It suffices to factor $s$ as a cofibration followed by a trivial fibration $w$,
  and take a fibrant replacement $z$ of the intermediate object as in the following diagram:
  \begin{equation}
    \begin{tikzcd}
      F \arrow[r, "s"] \arrow[rd] \arrow[d, swap, "1"] & G \\
      F \arrow[rd, swap, "{s'}"] & G_1 \arrow[u, swap, "w"] \arrow[d, "z"] \\
      & G'
    \end{tikzcd}
  \end{equation}
  In the homotopy category,
  the roof $zw^{-1}\colon G \to G'$
  is equivalent to
  a homotopy equivalence $h \colon G \to G'$,
  since the objects $G$ and $G'$ are in $\operatorname{Ch}(\Inj\mathcal{C})$.
  Then $s'$ is the desired cofibration (i.e., degreewise monomorphism) homotopic to $hs$.
\end{proof}

Returning to the setup of \cref{theorem:right-adjoint-restriction},
by \cref{lemma:presentation}, and up to isomorphism in $\KKK(\Inj\mathcal{C})$,
we can write our morphism~$s\colon F\to G$ as
\begin{equation}
  s\colon J^\bullet\hookrightarrow I^\bullet,
\end{equation}
with~$s$ a degreewise monomorphism, and~$J^\bullet$ and~$I^\bullet$ complexes of injectives.

We will repeatedly use the following lemma for lifting injective objects along the functor~$\Hom_{\overline{R}}(R,-)$.

\begin{lemma}
  \label{lemma:unique-lift-for-injectives}
  For every object~$I \in \Inj \cC$
  there exists an object~$\Ibar \in \Inj \overline{\cC}$,
  unique up to isomorphism,
  such that~$\Hom_{\overline{R}}(R,\Ibar)\simeq I$.
\end{lemma}
\vspace{-.5\baselineskip}
\begin{proof}
  As~$\overline{\cC}$
  is flat over $ \overline{R} $, \cite[Proposition 3.4]{MR2238922} implies that any injective object of~$\overline{\cC}$
  is coflat (see, say, \cite[Definition 2.5]{MR2238922} for the definition of coflat objects).
  Hence it admits a lift to a coflat object of~$\overline{ \cC }$,
  which is unique up to isomorphisms by the deformation theory \cite[Theorem 6.11]{MR2175388}
  and the injectivity of $ I $.

  Now it remains to show that there is at least one lift~$\Ibar \in \Inj ( \overline{ \cC } )$
  of~$I$,
  which has to be the coflat lift. This is shown in the proof of \cite[Proposition 5.5]{MR2175388}.
\end{proof}

\begin{lemma}
  \label{lemma:crude-lifting}
  Suppose that there exists an object~$\overline{G} \in \KKK ( \Inj ( \overline{ \cC }))$
  such that
  \begin{equation}
    \KKK \left( \Hom_{\overline{R}}(R, - ) \right) ( \overline{ G } )
    \simeq
    J^\bullet
    \in
    \KKK ( \Inj ( \cC ) ).
  \end{equation}
  Then there exists a complex~$\smash{\overline{J}}^\bullet$
  of objects in~$\Inj ( \overline{ \cC } )$
  such that~$\overline{ G } \simeq \smash{\overline{J}}^\bullet\in\KKK ( \Inj ( \overline{ \cC }))$
  and for each~$i \in \bZ$
  one has~$\Hom_{ \overline{ R } } ( R, \overline{J}^i )\simeq J^i\in\Inj ( \cC )$.
\end{lemma}
\vspace{-.5\baselineskip}
\begin{proof}
  For each~$i \in \bZ$,
  let~$\smash{\overline{J}}^i \in \Inj ( \overline{ \cC } )$
  be the lift of~$J^i$
  whose existence is shown in \cref{lemma:unique-lift-for-injectives}.
  Now the assertion follows from the crude lifting lemma \cite[Proposition 4.3]{MR2175388}.
\end{proof}

From now on, we will fix a lift~$\smash{\overline{I}}^{\bullet}$
of~$I^\bullet$
as in \cref{lemma:crude-lifting} once for all. Also, for each~$i \in \bZ$
consider the lift~$\smash{\overline{J}}^i$
of~$J^i$.
Using the injectivity of~$I^i$,
it follows from the deformation theory \cite[Theorem 6.12]{MR2175388}
that~$s^i \colon J^i \hookrightarrow I^i$
admits \emph{at least one lift}\footnote{
  In fact, the zeroth cohomology in \cite[Theorem~6.12(2)]{MR2175388}
  is nonzero as soon as~$J^i$ is nonzero.
  Hence the lift is not unique in general.
}~$\smash{\overline{ s }}^i \colon \smash{\overline{J}}^i \to \smash{\overline{ I }}^i$.
Note that~$\smash{\overline{ s }}^i$
is split injective. In fact, the injectivity of~$I^i$
implies that~$s^i$
admits a left inverse~$t^i \colon I^i \to J^i$.
By the same deformation theory as before,
it admits a lift~$\smash{\overline{ t }}^i \colon \smash{\overline{ I }}^i \to \smash{\overline{J}}^i$
and we know that
\begin{equation}
  n \defeq\smash{\overline{ t }}^i\circ \smash{\overline{ s }}^i - \identity_{ \smash{\overline{J}}^i }
  \in
  \End_{ \overline{ \cC } } ( \smash{\overline{J}}^i )
\end{equation}
is actually coming from
its subspace~$\Hom_{\overline{\cC}} \left( \Hom_{ \overline{R}} ( \mathfrak{b}, \smash{\overline{J}}^{i} ),\smash{\overline{J}}^i \right)$.

As explained in \cite[Proposition~5.4(i)]{MR2175388},
we have that~$\ker(\Hom_{\overline{R}}(R,-))^2=0$,
because of the standing assumption on the ring extensions,
which implies that~$\mathfrak{b}^2=0$,
so we are in the setting of \cite[\S3.2]{MR2175388}.

We have that~$n\in\ker(\Hom_{\overline{R}}(R,-))$,
because~$\Hom_{\overline{R}}(R,n)=t^i\circ s^i-\identity_{J^i}=0$.
Hence $n\circ n=0$ and~$( \identity_{\smash{\overline{J}}^i} - n )\circ \smash{ \overline{ t }}^i$
is a left inverse to~$\smash{\overline{ s }}^i$.
Now we fix for each~$i \in \bZ$
such a lift of~$s^i$
and name it~$\smash{\overline{ s }}_0^i$.
We also fix a left inverse~$\smash{\overline{ t }}_0^i$
to~$\smash{\overline{ s }}_0^i$.
This will give us, for each~$i \in \bZ$,
the direct sum decomposition
\begin{equation}
  \label{equation:decomposition-of-lift}
  \smash{\overline{I}}^i=\smash{\overline{J}}^i\oplus\smash{\overline{K}}^i,
\end{equation}
where~$\smash{\overline{J}}^i = \overline{ s }_0^i ( \smash{\overline{J}}^i )$
and~$\smash{\overline{K}}^i = \ker ( \smash{\overline{ t }_0^i} )$.
This brings us roughly in a situation like that of \cite[\S2.A.7]{MR2665168}, and we will combine the
ideas from op.~cit.~with the tools from \cite{MR2175388} to prove the result.

We consider the different (graded) lifts~$\overline{s}$ of $ s $
up to the action of the infinitesimal automorphisms of~$\smash{\overline{J}}^\bullet$.
These allow us to normalise the components
of~$\overline{s}$ such that they are of the form
\begin{equation}
  \overline{s}^i=\identity_{\smash{\overline{J}}^i}\oplus\beta_{\smash{\overline{s}}}^i
\end{equation}
for some
\begin{equation}
  \beta_{\overline{s}}^i
  \colon
  \Hom_{\overline{R}}( \mathfrak{b},\smash{\overline{J}}^i)\to\smash{\overline{K}}^i
\end{equation}
under the decomposition \eqref{equation:decomposition-of-lift}.
Via this normalisation, infinitesimal automorphisms between lifts
are forced to be the identity on the component~$\smash{\overline{J}}^\bullet$.

By combining the morphisms~$\overline{s}^i$ with the
embeddings~$\smash{\overline{K}}^i\hookrightarrow\smash{\overline{I}}^i$ we can identify~$\overline{s}$
with an automorphism~$b_{\overline{s}}$ of~$\smash{\overline{I}}^\bullet$ whose
components~$b_{\overline{s}}^i$ are of the form
\begin{equation}
  \label{equation:b}
  b_{\overline{s}}^i
  =
  \begin{pmatrix}
    \identity_{\smash{\overline{J}}^i} & 0 \\
    \beta_{\overline{s}}^i & \identity_{\smash{\overline{K}}^i}
  \end{pmatrix}
\end{equation}
The differential~$d_{\overline{I}^\bullet}$ can be decomposed using \eqref{equation:decomposition-of-lift}
into~$
\begin{psmallmatrix} d_{1,1} & d_{1,2} \\ d_{2,1} & d_{2,2}
\end{psmallmatrix}$,
and using~$b_{\overline{s}}$ it can be conjugated
into~$\tilde{d}_{\overline{s}}=b_{\overline{s}}^{-1}\circ d_{\smash{\overline{I}}^\bullet}\circ b_{\overline{s}}$,
whose entry in position~$(2,1)$ under the decomposition \eqref{equation:decomposition-of-lift}
we will denote by~$\varphi_{\overline{s}}$.
More explicitly, we have that
\begin{equation}
  \label{equation:varphi-description}
  \varphi_{\overline{s}}
  =
  -\beta_{\overline{s}}\circ d_{1,1} + d_{2,1} + d_{2,2}\circ\beta_{\overline{s}},
\end{equation}
as~$\beta_{\overline{s}}\circ d_{1,2}\circ\beta_{\overline{s}}=0$,
because~$\beta_{\overline{s}}\in\ker(\Hom_{\overline{R}}(R,-))$ and we have
that~$\ker(\Hom_{\overline{R}}(R,-))^2=0$ as before. This discussion proves the following lemma.

\begin{lemma}
  \label{lemma:subcomplex-iff-varphi-0}
  The map~$\overline{s}$ embeds~$\smash{\overline{J}}^\bullet $ into~$ \smash{\overline{ I }}^\bullet $ as a subcomplex
  if and only if~$\varphi_{\overline{s}}=0$.
\end{lemma}

Our next step is to analyse the differential~$\tilde{d}_{\overline{s}}$ as an endomorphism of degree~1
of~$\smash{\overline{I}}^\bullet$. For this we consider the short exact sequence of complexes of vector spaces
\begin{equation}
  0
  \to\Hom_-^\bullet(\Hom_{\overline{R}}( \mathfrak{b},\smash{\overline{I}}^\bullet),\smash{\overline{I}}^\bullet)
  \to\Hom^\bullet  (\Hom_{\overline{R}}( \mathfrak{b},\smash{\overline{I}}^\bullet),\smash{\overline{I}}^\bullet)
  \to\Hom_+^\bullet(\Hom_{\overline{R}}( \mathfrak{b},\smash{\overline{I}}^\bullet),\smash{\overline{I}}^\bullet)
  \to 0
  \label{equation:ses-Hom}
\end{equation}
where we define~$\Hom_-^\bullet(\Hom_{\overline{R}}(
\mathfrak{b},\smash{\overline{I}}^\bullet),\smash{\overline{I}}^\bullet)$ as the subcomplex of
morphisms~$\smash{\overline{I}}^i\to\smash{\overline{I}}^{i+j}$ of degree~$j$ which are zero
on~$\smash{\overline{J}}^i$ after reduction to~$R$.

By the construction we have the following lemma.
\begin{lemma}
  \label{lemma:1-cocycle}
  The differential~$\tilde{d}_{\overline{s}}$ defines a~1\dash cocycle~$
  \begin{psmallmatrix} 0 & 0 \\ \varphi & 0
  \end{psmallmatrix}$ of~$\Hom_+^\bullet(\Hom_{\overline{R}}(
  \mathfrak{b},\smash{\overline{I}}^\bullet),\smash{\overline{I}}^\bullet)$.
\end{lemma}
We are now in the position to define the obstruction class.
\begin{definition}
  The \emph{obstruction class} for~$s\colon F\to G$ is
  \begin{equation}
    \obstr(s)\defeq[\tilde{d}_{\overline{s}}]\in\HH^1(\Hom_+^\bullet(\Hom_{\overline{R}}(
    \mathfrak{b},\smash{\overline{I}}^\bullet),\smash{\overline{I}}^\bullet)).
  \end{equation}
\end{definition}
In the definition of the obstruction class there were choices involved, but by the proof of
\cref{proposition:torsor} these do not influence the cohomology class.

We now check that it really serves the role of an obstruction class.
\begin{proposition}
  \label{proposition:o-is-obstruction}
  The class~$\obstr(s)$ is the obstruction to the lifting of~$s$ as a morphism in the homotopy category of injectives.
  Namely,~$\obstr(s) = 0$
  if and only if there exists a morphism of complexes~$\overline{ s } \colon \smash{\overline{J}}^\bullet\to\smash{\overline{I}}^\bullet$
  such that~$\KKK ( \Hom_{ \overline{R} } ( R, - ) ) ( \overline{ s } )=s$.
\end{proposition}

\begin{proof}
  Assume that a lift~$\overline{s}$ (as a morphism of cochain complexes) exists.
  By \cref{lemma:subcomplex-iff-varphi-0} we have that~$\varphi_{\overline{s}}=0$, as
  then~$\smash{\overline{J}}^\bullet$ is a subcomplex of~$\smash{\overline{I}}^\bullet$ via $ \overline{ s } $.

  Conversely, if~$\obstr(s)=0$, then there exists a morphism
  \begin{equation}
    \psi=
    \begin{pmatrix}
      0 & 0 \\
      \psi_{2,1} & 0
    \end{pmatrix}
  \end{equation}
  such that
  \begin{equation}
    [\tilde{d}_{\overline{s}}]=\delta([\psi]),
  \end{equation}
  where~$\delta$ is the differential in the complex~$\Hom_+^\bullet(\Hom_{\overline{R}}(
  \mathfrak{b},\smash{\overline{I}}^\bullet),\smash{\overline{I}}^\bullet)$. This condition can be rephrased as
  \begin{equation}
    \label{equation:varphi-condition}
    \varphi=d_{2,2}\circ\psi_{2,1}-\psi_{2,1}\circ d_{1,1}.
  \end{equation}
  Then we claim that $\smash{\overline{J}}^\bullet$ becomes a subcomplex of~$\smash{\overline{I}}^\bullet$ via
  \begin{equation}
    \overline{s}\defeq\identity_{\smash{\overline{J}}^\bullet}\oplus(\beta-\psi_{2,1}).
  \end{equation}
  To prove this claim, it suffices by \cref{lemma:subcomplex-iff-varphi-0} to compute the
  component~$\varphi_{\overline{s}}$ of~$\tilde{d}$ associated to~$\beta-\psi_{2,1}$. Starting from the
  description in \eqref{equation:varphi-description} we can regroup and apply
  \eqref{equation:varphi-condition} to see that
  \begin{equation}
    \begin{aligned}
      &-(\beta-\psi_{2,1})\circ d_{1,1} + d_{2,1} + d_{2,2}\circ(\beta-\psi_{2,1}) \\
      &\qquad=(-\beta\circ d_{1,1} + d_{2,1} + d_{2,2}\circ\beta) - (d_{2,2}\circ\psi_{2,1} -
      \psi_{2,1}\circ d_{1,1}) \\
      &\qquad=0.
    \end{aligned}
  \end{equation}
  This completes the proof.
\end{proof}

This doesn't quite prove \cref{item:right-adjoint-obstruction} of \cref{theorem:right-adjoint-restriction} yet,
as the obstruction class lives in an a priori different cohomology space.
This will be remedied in \cref{proposition:hom-plus-is-hom}.
Before doing so, we first prove how the set of lifts is a torsor.
\begin{proposition}
  \label{proposition:torsor}
  Assume that~$\HH^{-1}(\RRHom_{\mathcal{C}_0}(\RRHom_{R_0}( \mathfrak{b},F_0),H_0))=0$. If~$\obstr(s)=0$,
  then the set of isomorphism classes of lifts is
  a torsor under $\HH^0(\Hom_+^\bullet(\Hom_{\overline{R}}(\mathfrak{b},\smash{\overline{I}}^\bullet),\smash{\overline{I}}^\bullet))$.
\end{proposition}

\begin{proof}
  Let~$\beta\defeq\beta_{\overline{s}}$ and~$\beta'\defeq\beta_{\overline{s}'}$ be associated to
  different lifts~$\overline{s}$ and~$\overline{s}'$ of~$s$. By \cref{lemma:subcomplex-iff-varphi-0}
  this means that the associated~$\varphi_{\overline{s}}$ and~$\varphi_{\overline{s}'}$ are zero, i.e.,
  \begin{equation}
    -\beta\circ d_{1,1}+d_{2,1}+d_{2,2}\circ\beta=0
  \end{equation}
  and
  \begin{equation}
    -\beta'\circ d_{1,1}+d_{2,1}+d_{2,2}\circ\beta'=0
  \end{equation}
  But then~$\beta-\beta'$ is a~$0$-cocycle in~$\Hom_+^\bullet(\Hom_{\overline{R}}(
  \mathfrak{b},\smash{\overline{I}}^\bullet),\smash{\overline{I}}^\bullet)$, which is analogous to what
  happened in \eqref{equation:varphi-condition}.

  Conversely, if~$\beta\defeq\beta_{\overline{s}}$ is associated to a lift~$\overline{s}$ of~$s$,
  and~$\xi$ is a~$0$\dash cocycle in~$\Hom_+(\Hom_{\overline{R}}(
  \mathfrak{b},\smash{\overline{I}}^\bullet),\smash{\overline{I}}^\bullet)$, then~$\beta+\xi$ defines
  another lift of~$s$, because
  \begin{equation}
    \begin{aligned}
      &-(\beta+\xi)\circ d_{1,1}+d_{2,1}+d_{2,2}\circ(\beta+\xi) \\
      &\qquad=(-\beta\circ d_{1,1}+d_{2,1}+d_{2,2}\circ\beta)+\delta(\xi) \\
      &\qquad=0
    \end{aligned}
  \end{equation}

  Finally, note that~$\beta$ and~$\beta'$ define the same morphism from~$\smash{\overline{J}}^\bullet$
  to~$\smash{\overline{I}}^\bullet$ if and only if~$b_\beta=b_{\beta'}$ (recall \eqref{equation:b}) as
  endomorphisms of~$\smash{\overline{I}}^\bullet$. But~$b_{\beta}-b_{\beta'}$ is a~0\dash coboundary
  of~$\Hom^\bullet(\smash{\overline{I}}^\bullet,\smash{\overline{I}}^\bullet)$ if and
  only~$b_\beta-b_{\beta'}$ is a~0\dash coboundary of~$\Hom_+^\bullet(\Hom_{\overline{R}}(
  \mathfrak{b},\smash{\overline{I}}^\bullet),\smash{\overline{I}}^\bullet)$. But this is the case if and
  only if there exists~$\eta\in\Hom^{-1}(\smash{\overline{J}}^\bullet,\smash{\overline{K}}^\bullet)$ such
  that $\beta-\beta'=d_{2,2}\circ\eta-\eta\circ d_{1,1}$.
\end{proof}

Finally we prove the following identification, finishing the proof of \cref{theorem:right-adjoint-restriction}.

\begin{proposition}
  \label{proposition:hom-plus-is-hom}
  There exist isomorphisms of complexes
  \begin{equation}
    \Hom_+^\bullet(\Hom_{\overline{R}}( \mathfrak{b},\smash{\overline{I}}^\bullet),\smash{\overline{I}}^\bullet)
    \simeq\Hom_+^\bullet(\Hom_{R_0}( \mathfrak{b},I_0^\bullet),I_0^\bullet)
    \simeq\Hom^\bullet(\Hom_{R_0}( \mathfrak{b},K_0^\bullet),J_0^\bullet),
  \end{equation}
  where as before~$K_0^\bullet\defeq I_0^\bullet/J_0^\bullet$.
\end{proposition}

\begin{proof}
  The first isomorphism follows from the fact that~$\overline{R}\to R$ is small relative to~$R\to R_0$,
  i.e.~that~$ \mathfrak{a} \mathfrak{b}=0$. We will denote the image of an element~$(s^i)_i$ by~$(t^i)_i$.

  Now define
  \begin{equation}
    \mu^j\colon\Hom_+^j(\Hom_{R_0}( \mathfrak{b},I_0^\bullet),I_0^\bullet)=\frac{\Hom^j(\Hom_{R_0}(
    \mathfrak{b},I_0^\bullet),I_0^\bullet)}{\Hom_-^j(\Hom_{R_0}(
    \mathfrak{b},I_0^\bullet),I_0^\bullet)}\to\Hom^j(\Hom_{R_0}( \mathfrak{b},K_0^\bullet),J_0^\bullet)
  \end{equation}
  by sending~$t^j$ to the morphism~$u^j$ defined by the composition
  \begin{equation}
    \begin{tikzcd}
      \Hom_{R_0}( \mathfrak{b},J_0^i) \arrow[r, "u^j"] \arrow[d, hook] & K_0^{i+j} \\
      \Hom_{R_0}( \mathfrak{b},I_0^i) \arrow[r, "t^j"] & I_0^{i+j} \arrow[u, two heads]
    \end{tikzcd}
  \end{equation}
  giving the second isomorphism.
\end{proof}

\begin{proof}[Proof of \cref{theorem:right-adjoint-restriction}]
  To prove \cref{item:right-adjoint-obstruction,item:right-adjoint-torsor}
  it suffices to apply the identification from
  \cref{proposition:hom-plus-is-hom}
  to \cref{proposition:o-is-obstruction,proposition:torsor}.
\end{proof}

For the geometric applications we have in mind
we need the following result, dual to \cref{theorem:right-adjoint-restriction}.
To deduce it from \cref{theorem:right-adjoint-restriction},
it suffices to consider the opposite categories of~$\mathcal{C}_0$, $\mathcal{C}$
and~$\overline{\mathcal{C}}$ in \eqref{equation:deformations},
using \cite[Proposition~8.7(3)]{MR2238922}.
This exchanges left and right adjoints, and turns co-Grothendieck categories into Grothendieck categories.
We thus obtain the following.
\begin{theorem}[Dual version of \cref{theorem:right-adjoint-restriction}]
  \label{theorem:left-adjoint-restriction}
  Let~$\mathcal{C}_0$ be a co-Grothendieck abelian category.
  Consider an exact triangle
  \begin{equation}
    F_0\overset{s_0}{\to}G_0\to H_0\to F_0[1]
  \end{equation}
  in~$\KKK(\Proj\mathcal{C}_0)$, and let
  \begin{equation}
    F\overset{s}{\to}G
  \end{equation}
  be a morphism in~$\KKK(\Proj\mathcal{C})$ which restricts to~$s_0$, i.e.,~$\KKK(R_0\otimes_R s)=s_0$.

  Let~$\overline{G}$ be a lift of~$G$ from~$\KKK(\Proj\mathcal{C})$ to~$\KKK(\Proj\overline{\mathcal{C}})$
  along~$\KKK(R\otimes_{\overline{R}}-)$.
  \begin{enumerate}
    \item
      There exists an obstruction class
      \begin{equation}
        \obstr(s)\in\HH^1(\RRHom_{\mathcal{C}_0}(F_0, \mathfrak{b}\otimes_{R_0}^\LL H_0))
      \end{equation}
      such that~$\obstr(s)=0$ if and only if~$s$ lifts to a
      morphism~$\overline{s}\colon\overline{F}\to\overline{G}$ in~$\KKK(\Proj\overline{\mathcal{C}})$.
    \item
      Assume that~$\HH^{-1}(\RRHom_{\mathcal{C}_0}(F_0, \mathfrak{b}\otimes_{R_0}^\LL H_0))=0$.
      If~$\obstr(s)=0$,
      then the set of isomorphism classes of such lifts is a torsor under
      \begin{equation}
        \HH^0(\RRHom_{\mathcal{C}_0}(F_0, \mathfrak{b}\otimes_{R_0}^\LL H_0)).
      \end{equation}
  \end{enumerate}
\end{theorem}

\section{Restriction to the derived category}
\label{section:restriction}
In order to apply \cref{theorem:left-adjoint-restriction}
to the setting of interest to us
we need a comparison result as in \cite[\S6.3]{MR2175388}.
For geometric applications we are not so much interested in deforming morphisms in~$\KKK(\Proj\Pro\mathcal{C})$,
but rather in derived categories which avoid the Pro-construction.

We need the analogue of \cite[Proposition~6.2]{MR2175388}. Let us first recall the following definition
from \cite[Definition 6.1]{MR2175388}.
\begin{definition}
  \label{definition:L}
  A diagram of functors
  \begin{equation}
    \label{equation:L-condition}
    \begin{tikzcd}
      \mathcal{C} \arrow[r, "F"] \arrow[d, swap, "H"] & \mathcal{C}' \arrow[d, "H'"] \\
      \mathcal{D} \arrow[r, "G"] & \mathcal{D}'
    \end{tikzcd}
  \end{equation}
  \emph{satisfies (L)} if the following conditions hold:
  \begin{enumerate}
    \item the diagram is commutative up to natural isomorphism;
    \item $F$ and $G$ are fully faithful;
    \item if $H'(C')\simeq G(D)$ for some~$C'\in\mathcal{C}'$ and~$D\in\mathcal{D}$, then there exists an
      object~$C\in\mathcal{C}$ such that~$C'\simeq F(C)$.
  \end{enumerate}
\end{definition}
Here~$H$ and~$H'$ are to be interpreted as the restriction functors for two deformations.

Next, let~$f\colon D_1\to D_2$ be a morphism in~$\mathcal{D}_2$,
and let~$d \colon D_{ 2} \simto H( \overline{D}_2 )$
be a lift of~$D_2$.
We introduce the following variation on \cite[Definition~3.1(2)]{MR2175388}.
\begin{definition}
  The \emph{set of lifts of~$f$},
  denoted by~$\lift_H(f\mid d)$,
  is the set of isomorphism classes of
  the groupoid where an object is a pair~$(d_1,\overline{f})$ of
  \begin{enumerate}
    \item a lift~$d_1 \colon D_1 \simto H ( \overline{ D }_1 )$
      of~$D_1$,
    \item a morphism~$\overline{f} \colon \overline{ D }_1 \to \overline{ D }_2$,
  \end{enumerate}
  such that~$H ( \overline{f} ) d_1=d f$,
  and a morphism from~$(d_1,\overline{f})$
  to~$(d_1' \colon D_1 \simto H ( \smash{\overline{D}}_1 ' ),\smash{\overline{f}}' \colon \smash{\overline{D}}_1' \to \overline{D}_2)$
  is an isomorphism~$\varphi \colon \overline{ D }_1 \simto\smash{\overline{D}}_1 '$
  such that~$\smash{\overline{f}}' \varphi \overline{f} = \overline{f}$
  and~$H ( \varphi ) d_1 = d_1 '$.
\end{definition}

Suppose a diagram of the form \eqref{equation:L-condition} satisfies (L) from \cref{definition:L}.
Let~$f\colon D_1\to D_2$ be a morphism in~$\mathcal{D}$,
let~$ c_2 \colon D_2 \simto H ( C_2 )\in\lift_H(D_2)$
be an element of the set of lifts of~$D_2$
in the sense of \cite[Definition~3.1(1)]{MR2175388},
and let~$C_2'\defeq F(C_2)$.
Let~$g \defeq G ( f )$, and
\begin{equation}
  c '_2 \defeq G ( c_2 )
  \colon
  G ( D_2 ) \simto F ( C '_2 ).
\end{equation}
Here we used the isomorphism~$G H \simeq H ' F$.
Consider the canonical map
\begin{equation}
  \label{equation:map-between-sets-of-lifts}
  \lift_H(f\mid c_2) \to \lift_{H'}(g\mid c_2')
\end{equation}
which sends a lift~$(c_1 \colon D_1 \simto H ( C_1 ),\overline{f} \colon C_1 \to C_2)$
of~$f$
to the following lift of $g$.
\begin{equation}
  (
    G ( c_1 ) \colon G ( D_1 ) \simto G H ( C_1 ) \simeq H ' F ( C_1 ),
    F ( \overline{f} ) \colon F ( C_1 ) \to F ( C_2 )
  ).
\end{equation}
Analogous to \cite[Proposition~6.2]{MR2175388} one obtains the following proposition from the property (L).
\begin{proposition}
  \label{proposition:comparison-of-lift-groupoids}
  The map \eqref{equation:map-between-sets-of-lifts} is a bijection.
\end{proposition}

We wish to apply this to the dual setting of \cite[\S6.3]{MR2175388}.
For this, let us consider the setting of flat abelian deformations of \eqref{equation:deformations},
without any restriction on the category.
The dual of \cite[Proposition~6.5]{MR2175388} in our lifting problem for morphisms is the following.

\begin{proposition}
  There is a diagram
  \begin{equation}
    \begin{tikzcd}
      \derived^-(\Pro\overline{\mathcal{C}}) \arrow[d] & \KKK^-(\Proj\Pro\overline{\mathcal{C}})
      \arrow[r, hook] \arrow[l, "\simeq"] \arrow[d] & \KKK(\Proj\Pro\overline{\mathcal{C}}) \arrow[d] \\
      \derived^-(\Pro\mathcal{C}) & \KKK^-(\Proj\Pro\mathcal{C}) \arrow[r, hook] \arrow[l, "\simeq"] &
      \KKK(\Proj\Pro\mathcal{C})
    \end{tikzcd}
  \end{equation}
  where the vertical arrows are the appropriately defined restriction functors. Both squares satisfy (L)
  (see \cref{definition:L}), hence the deformation theory for morphisms with fixed lift of the
  target from \cref{theorem:left-adjoint-restriction} restricts to lifting morphisms in the
  category~$\derived^-(\mathcal{C})$.
\end{proposition}

For our purposes we mostly care about perfect complexes on schemes,
in the setting of \cref{example:geometric-setting}.
This requires a further restriction step, which is given by the following proposition.
We will use the notion of Tor-dimension and pseudo-coherent complexes,
see, e.g.,
\cite[\href{https://stacks.math.columbia.edu/tag/08CG}{Tag 08CG}]{stacks-project} and
\cite[\href{https://stacks.math.columbia.edu/tag/08CB}{Tag 08CB}]{stacks-project}.

\begin{proposition}
  \label{proposition:perfect-complexes}
  Let~$\overline{X}, X$ be as in \cref{example:geometric-setting}. There is a diagram
  \begin{equation}
    \begin{tikzcd}
      \Perf\overline{X} \arrow[r, hook] \arrow[d] & \derived_\fTd^-(\Qcoh\overline{X}) \arrow[d]
      \arrow[r, "\cong"] & \derived_{\fTd,\Qcoh\overline{X}}^-(\Pro\Qcoh\overline{X}) \arrow[d] \arrow[r,
      hook] & \derived^-(\Pro\Qcoh\overline{X}) \arrow[d] \\
      \Perf X \arrow[r, hook] & \derived_\fTd^-(\Qcoh X) \arrow[r, "\cong"] & \derived_{\fTd,\Qcoh
      X}^-(\Pro\Qcoh X) \arrow[r, hook] & \derived^-(\Pro\Qcoh X)
    \end{tikzcd}
  \end{equation}
  where the vertical arrows are the appropriately defined restriction functors. All three squares satisfy
  (L) (see \cref{definition:L}).
\end{proposition}

\begin{proof}
  The fact that the middle and the right square satisfy (L) is dual to \cite[Proposition~6.9]{MR2175388}.

  For the left square we use the characterization of perfect complexes
  as pseudo-coherent complexes of finite Tor-dimension,
  from \cite[\href{https://stacks.math.columbia.edu/tag/08CQ}{Tag 08CQ}]{stacks-project}.
  To prove that the left square satisfies~(L),
  we consider an object~$\overline{E}$ in~$\derived_\fTd^-(\Qcoh\overline{X})$,
  such that its restriction~$R\otimes_{\overline{R}}^\LL\overline{E}$
  lies in the essential image from the inclusion of~$\Perf X$.
  It suffices to prove that~$\overline{E}$ is pseudo-coherent.

  To see this we use the~2-out-of-3 property for pseudo-coherence from
  \cite[\href{https://stacks.math.columbia.edu/tag/08CD}{Tag 08CD}]{stacks-project}, and consider the triangle
  \begin{equation}
    \mathfrak{b}\otimes_{\overline{R}}^\LL\overline{E}\to\overline{E}\to
    R\otimes_{\overline{R}}\overline{E}\to\mathfrak{b}\otimes_{\overline{R}}^\LL\overline{E}[1]
  \end{equation}
  induced by taking the derived tensor product the short exact
  sequence~$0\to\mathfrak{b}\to\overline{R}\to R\to 0$ with~$\overline{E}$. By assumption we have
  that~$R\otimes_{\overline{R}}\overline{E}$ is perfect, hence pseudo-coherent.

  To see that~$\mathfrak{b}\otimes_{\overline{R}}^\LL \overline{E}$ is pseudo-coherent, we consider the isomorphism
  \begin{equation}
    \mathfrak{b}\otimes_{\overline{R}}^\LL\overline{E}
    \cong
    \mathfrak{b}\otimes_R^\LL(R\otimes_{\overline{R}}^\LL\overline{E})
  \end{equation}
  using the fact that~$\mathfrak{b}^2=0$ to conclude that~$\overline{E}$ is itself pseudo-coherent, hence
  it lies in the essential image of the inclusion of~$\Perf\overline{X}$ in~$\derived_\fTd^-(\Qcoh\overline{X})$.
\end{proof}

\begin{proof}[Proof of \cref{corollary:lifting-morphisms-of-perfect-complexes}]
  This follows from \cref{theorem:left-adjoint-restriction}
  and \cref{proposition:perfect-complexes}.
\end{proof}

We end this section with the following remarks.
\begin{remark}
  \label{remark:RHom-over-field}
  In the setting where~$R_0$ is semisimple, we can make the following simplification: we have an isomorphism
  \begin{equation}
    \RRHom_{X_0}(F_0, \mathfrak{b} \otimes_{R_0}^\LL H_0) \simeq \RRHom_{X_0}(F_0,H_0) \otimes_{R_0} \mathfrak{b},
  \end{equation}
  and the computations reduce to determining the usual Ext-groups of the objects which are involved.
\end{remark}

\begin{remark}
  \label{remark:recovering-lieblich}
  As explained in \cref{remark:recover-deformation-theory-complexes} we can recover the deformation
  theory of complexes from our deformation theory of complexes with fixed lift of the target. Namely we
  recover \cite[Theorem~3.1.1]{MR2177199} (for perfect complexes) as in \cref{remark:recover-deformation-theory-complexes} by taking~$G_0=0$ and using the zero object as the
  fixed lift of the target.
  Then condition in \cref{item:perfect-torsor} of \cref{corollary:lifting-morphisms-of-perfect-complexes}
  translates to the gluability condition in \cite{MR2177199}.
\end{remark}

\section{Uniqueness of deformations of semiorthogonal decompositions}
\label{section:unique-deformation-decomposition-triangles}
As mentioned in the introduction,
our initial motivation for developing the deformation theory in \cref{section:KInj-deformations,section:restriction}
was to prove that semiorthogonal decompositions deform uniquely in smooth and proper families of schemes.
We will establish this uniqueness of deformations of semiorthogonal decompositions
as a uniqueness result for
deformations of decomposition triangles.

First we recall some following notation,
before we can state \cite[Definition~5.2]{2002.03303}.
For an object (or kernel)~$K\in\Perf\cX\times_U\cX$
we have the associated Fourier--Mukai functor
\begin{equation}
  \Phi_K\colon\Perf\cX\to\Perf\cX:E\mapsto\mathbf{R}\mathrm{pr}_{2,*}(\pr_1^*(E)\otimes^{\mathbf{L}}K),
\end{equation}
and we denote its essential image by~$\cE_K\subseteq\Perf\cX$.

Recall also from \cite[Lemma~2.1]{2002.03303}
that~$\mathcal{O}_{\Delta_f}\in\Perf\cX\times_U\cX$
if~$f\colon\cX\to U$ is a smooth and separated morphism
of quasicompact and semiseparated schemes.
If~$\langle\cA,\cB\rangle$ is a~$U$-linear semiorthogonal decomposition
of~$\Perf\cX$,
then by base change of the semiorthogonal decomposition (see the proof of \cite[Lemma~5.4]{2002.03303})
one obtains a semiorthogonal decomposition of~$\Perf\cX\times_U\cX$.
Decomposing the structure sheaf of the diagonal
gives a distinguished triangle
\begin{equation}
  \begin{tikzcd}
    K_{\cB}\arrow{r} & \Ocal_{\Delta_{f}} \arrow{r} & K_{\cA} \arrow{r} & K_{\cB}[1],
  \end{tikzcd}
\end{equation}
for which~$\RRlHom(\cE_{K_{\cB}},\cE_{K_{\cA}})=\RRlHom(\cB,\cA)=0$
(meaning it is zero for all choices of objects in those categories),
which we refer to as the \emph{decomposition triangle}.
The correspondence between semiorthogonal decompositions
and decomposition triangles
is explained by \cite[Theorem~5.7]{2002.03303},
and this allows us to work with decomposition triangles in what follows.

\begin{definition}
  \label{definition:dec-functor}
  Let~$f\colon\cX\to U$ be a smooth and proper morphism of schemes.
  We define the functor~$\dec_{\Delta_f}\colon\Aff_U^{\op}\to\Sets$
  of \emph{decompositions of the diagonal}
  by sending an affine~$U$-scheme~$\phi\colon V\to U$
  to the set of distinguished triangles
  \begin{equation}
    \label{equation:decomposition-triangle}
    \begin{tikzcd}
      \zeta\colon\qquad K_2\arrow{r} & \Ocal_{\Delta_{f_V}} \arrow{r} & K_1 \arrow{r} & K_2[1]
    \end{tikzcd}
  \end{equation}
  in $\Perf \Ycal_V$, such that
  \begin{equation}
    \RRlHom_{f_V}(\mathscr E_{K_2},\mathscr E_{K_1}) = 0,
  \end{equation}
  up to the equivalence relation
  given by identifying \eqref{equation:decomposition-triangle}
  with
  \begin{equation}
    \begin{tikzcd}
      \zeta'\colon\qquad K'_2\arrow{r} & \Ocal_{\Delta_{f_V}} \arrow{r} & K'_1 \arrow{r} & K'_2[1]
    \end{tikzcd}
  \end{equation}
  whenever there is an isomorphism of distinguished triangles
  \begin{equation}
    \label{equation:isomorphism-triangles}
    \begin{tikzcd}
      K_2\arrow{r}\isoarrow{d} & \Ocal_{\Delta_{f_V}} \arrow{r}\arrow[equal]{d} & K_1 \arrow{r}\isoarrow{d} & K_2[1]\isoarrow{d} \\
      K'_2\arrow{r} & \Ocal_{\Delta_{f_V}} \arrow{r} & K'_1 \arrow{r} & K'_2[1]
    \end{tikzcd}
  \end{equation}
  where we insist that the middle objects are connected by the identity morphism.
\end{definition}

\begin{remark}\label{remark:uniqueness of the isomorphism}
  The isomorphisms
  \(
    K _{ i } \simto K ' _{ i }
  \)
  for \( i = 1, 2 \)
  in \eqref{equation:isomorphism-triangles}
  are unique if they exist. Indeed, this follows from \cite[Lemma~5.5]{2002.03303}
  which implies that a decomposition of the diagonal coincides with
  the decomposition triangle associated to the base change of the semiorthogonal decomposition
  \(
    \langle
    \cE_{K_{\cA}},
    \cE_{K_{\cB}}
    \rangle
  \); note that the decomposition of an object with respect to a semiorthogonal decomposition is unique up to unique isomorphism.
\end{remark}

The proof of \cref{theorem:uniqueness-deformation-decomposition-triangles}
consists of two steps:
a deformation argument (using \cref{corollary:lifting-morphisms-of-perfect-complexes})
and an algebraization argument.

\begin{proof}[Proof of \cref{theorem:uniqueness-deformation-decomposition-triangles}]
  By replacing $f$ with its base change by~$\Spec R \to U$,
  we may without loss of generality assume that $U = \Spec R$.
  To emphasise the base, we will write~$f_R\colon\cX_R\to\Spec R$.
  We will also write~$\Ycal_R\defeq\Xcal\times_R\Xcal$
  and~$g_R\colon\Ycal\to\Spec R$.

  Let us define the inverse $\vartheta^{-1}$.
  An element
  \begin{equation}
    \xi_{\mathfrak m} \in \dec_{\Delta_f}(\Spec \bfk \to U)
  \end{equation}
  corresponds to a semiorthogonal decomposition
  \begin{equation}
    \label{equation:A-B-sod}
    \derived^\bounded(X) = \Perf X = \langle \cA,\cB \rangle
  \end{equation}
  for the $\bfk$-scheme $X = \cX\times_U\Spec \bfk$.
  This determines the decomposition triangle
  \begin{equation}
    \label{equation:starting-decomposition-triangle}
    \begin{tikzcd}
      K_{\cB} \arrow{r}{s_0} &  \Ocal_{\Delta_X} \arrow{r} & K_{\cA}
    \end{tikzcd}
  \end{equation}
  in~$\derived^\bounded(X\times_{ \bfk } X)=\Perf (X\times_{ \bfk } X)$.
  For $n\geq 0$ set $R_n = R/\mathfrak m^{n+1}$.
  We consider the base change diagrams
  \begin{equation}
    \label{equation:cartesian-diagram-Xn}
    \begin{tikzcd}
      \Xcal_{R_n} \arrow[hook]{r} \arrow[swap, "f_{R_n}"]{d} &
      \Xcal_R \arrow["f_R"]{d} \\
      \Spec R_n \arrow[hook]{r} &
      \Spec R
    \end{tikzcd}
  \end{equation}
  and
  \begin{equation}
    \label{equation:cartesian-diagram-Yn}
    \begin{tikzcd}
      \Ycal_{R_n} \arrow[hook]{r} \arrow[swap, "g_{R_n}"]{d} &
      \Ycal_R \arrow["g_R"]{d} \\
      \Spec R_n \arrow[hook]{r} &
      \Spec R
    \end{tikzcd}
  \end{equation}
  induced by the closed immersion $\Spec R_n \into \Spec R$.
  Since we know that
  \begin{equation}\label{equation:vanishing on the central fiber}
    \RRHom_{X\times_\bfk X}^\bullet(K_{\cB},K_{\cA}) = 0
  \end{equation}
  by the semiorthogonality \eqref{equation:A-B-sod},
  starting with $n=0$ we can inductively apply \cref{corollary:lifting-morphisms-of-perfect-complexes}
  (see also \cref{remark:RHom-over-field}) to
  \begin{equation}
    (\overline{R} \onto R \onto R_0) = (R_{n+1}\onto R_n\onto R_0 = \bfk)
  \end{equation}
  and the morphism~$K_\cB\to\Ocal_{\Delta_X}$
  from \eqref{equation:starting-decomposition-triangle},
  to obtain the existence of a unique lift $s_{n+1}$ of $s_n$.
  The fixed lift $\overline{G}$ in \cref{corollary:lifting-morphisms-of-perfect-complexes}
  i (at the $n$th step) taken to be the complex
  \begin{equation}
    \Ocal_{\Delta_{n+1}} \defeq \Ocal_{\Delta_f}\big|_{R_{n+1}} = \Ocal_{\Delta_{f_{n+1}}} \in \Perf \Ycal_{R_{n+1}}.
  \end{equation}
  Thus for each $n$,
  we obtain a morphism $s_n\colon K _{ \cB _{ R _{ n } } } \to \Ocal_{\Delta_{f_n}}$
  which is a deformation of $s_0$. Let
  \(
    K _{ \cA _{ R _{ n } } }
  \)
  be the cone of the morphism~$s_n$,
  i.e.,
  \begin{equation}\label{equation:n-th deformation}
    K _{ \cB _{ R _{ n } } }
    \to
    \Ocal_{\Delta_{f_n}}
    \to
    K _{ \cA _{ R _{ n } } }
    \to
    K _{ \cB _{ R _{ n } } } [1].
  \end{equation}
  We wish to verify it is again a decomposition of the diagonal.
  We have that
  \begin{equation}
    \label{equation:n-th-RHom-vanishing}
    \RRlHom_{f_{ R _{ n } }}(\mathscr E_{K_{\cB_{R _{ n }}}},\mathscr E_{K_{\cA_{R _{ n }}}}) = 0,
  \end{equation}
  as by the derived Nakayama lemma
  (see, e.g., \cite[Lemma~2.2]{2002.03303})
  it is enough to show the vanishing
  \begin{equation}
    \RRlHom_{f_{ R _{ n } }}(\mathscr E_{K_{\cB_{ R _{ n } }}},\mathscr E_{K_{\cA_{ R _{ n } }}})\otimes ^{ \mathbf{L} } _{ R _{ n } }\bfk = 0.
  \end{equation}
  Since $f_{ R _{ n } }$ is flat, the cartesian diagram
  \begin{equation}
    \label{equation:cartesian-diagram}
    \begin{tikzcd}
      \cX_{ \bfk } \MySymb{dr} \arrow[r, "\iota"] \arrow[d, "f_{ \bfk }"'] & \cX_{ R _{ n } } \arrow[d, "f_{ R _{ n } }"]\\
      \Spec \bfk \arrow{r} & \Spec R _{ n }
    \end{tikzcd}
  \end{equation}
  (which is also the case~$n=0$ of \eqref{equation:cartesian-diagram-Xn})
  is exact.
  Hence we can compute
  \begin{equation}
    \begin{aligned}
      \RRlHom_{f_{ R _{ n } }}(\mathscr E_{K_{\cB_{ R _{ n } }}},\mathscr E_{K_{\cA_{ R _{ n } }}})
      \otimes ^{ \mathbf{L} } _{ R _{ n } }\bfk
      &\simeq
      \mathbf{R} f_{ \bfk,\ast } \mathbf{L} \iota^{ \ast }
      \RRlHom_{\Xcal_{ R _{ n } }}(\mathscr E_{K_{\cB_{ R _{ n } }}},\mathscr E_{K_{\cA_{ R _{ n } }}}) \\
      &\simeq
      \RRHom_{ \cX_{ \bfk } }
      \left( \mathscr E_{K_{\cB_{ R _{ n } }}} \big|_{ \cX_{ \bfk } },\mathscr E_{K_{\cA_{ R _{ n } }}} \big|_{ \cX_{ \bfk } } \right).
    \end{aligned}
  \end{equation}
  Again by using the exact cartesian diagram \eqref{equation:cartesian-diagram}, one can confirm that
  \begin{equation}
    \mathscr E_{K_{\cB_{ R _{ n } }}}\big|_{ \cX_{ \bfk } }
    \subseteq
    \mathscr E_{ K_{ \cB } }
    =
    \cB,
    \quad
    \mathscr E_{K_{\cA_{ R _{ n } }}} \big|_{ \cX_{ \bfk } }
    \subseteq
    \mathscr E_{ K_{ \cA } }
    =
    \cA.
  \end{equation}
  Thus we obtain the desired vanishing \eqref{equation:n-th-RHom-vanishing}.
  Therefore the distinguished triangle \eqref{equation:n-th deformation}
  represents an $R _{ n }$-valued point of the functor $\dec_{\Delta_f}$.

  To finish the proof, we need to algebraize the formal deformation~$(s_n)_{n\geq 0}$.
  From the sequence~$(s_n)_{n\geq 0}$
  we obtain the sequence of deformations $(K_{n})_{n\geq 0}$ of $K_{\cB}$.
  By \cite[Proposition~3.6.1]{MR2177199},
  there exists a perfect complex $K_{\cB_{R}}$ on $\Ycal_R$
  with compatible isomorphisms $K_{\cB_{R}}|_{\Ycal_{R_{n}}}\, \widetilde{\to}\,K_{n}$.

  Next,
  we need to lift the sequence~$(s_n)_{n\geq0}$ to a morphism
  \begin{equation}
    \label{equation:morphism-over-R}
    \begin{tikzcd}
      s\colon K_{\cB_R} \arrow{r} & \Ocal_{\Delta_R},
    \end{tikzcd}
  \end{equation}
  where the target is the structure sheaf of the diagonal $\Xcal_R \hookrightarrow \Ycal_R$,
  and verify that we obtain an~$R$-valued point of~$\dec_{\Delta_f}$.
  To construct the morphism~$s$,
  consider the complex
  \begin{equation}
    \mathcal H = \RRlHom_{\Ycal_R}(K_{\cB_R},\Ocal_{\Delta_R}) \in \derived_{\coh}^{\textrm b}(\Ycal_R).
  \end{equation}
  Since $R$ is a complete noetherian local ring,
  we can identify $\coh(\Spf R)$ with $\coh(\Spec R)$.
  After this identification is made,
  we apply the comparison theorem in formal geometry
  to the complex $\mathcal H$:
  by \cite[Theorem~8.2.2]{MR2223409} and \cite[Remark~8.2.3 (a)]{MR2223409},
  the natural map \cite[(8.2.1.3)]{MR2223409} is an isomorphism.
  This means that we have an isomorphism
  \begin{equation}
    \label{equation:comparison-theorem}
    \begin{tikzcd}
      \mathrm{R}^0g_{R,\ast} \mathcal H \arrow{r}{\sim} & \varprojlim \mathrm{R}^0g_{R_n,\ast}\mathcal H_n,
    \end{tikzcd}
  \end{equation}
  where
  \begin{equation}
    \mathcal H_n = \mathcal H\big|_{Y_{R_n}} = \RRlHom_{\Ycal_{R_n}}(K_n,\Ocal_{\Delta_n}).
  \end{equation}
  Taking global sections
  in \eqref{equation:comparison-theorem} yields
  \begin{equation}
    \begin{aligned}
      \Hom_{\Ycal_R}(K_{\cB_R},\Ocal_{\Delta_R})
      &\simeq \Gamma(\Spec R,\varprojlim \mathrm{R}^0g_{R_n,\ast}\mathcal H_n) \\
      &\simeq \varprojlim \Gamma(\Spec R_n,\mathrm{R}^0g_{R_n,\ast}\RRlHom_{\Ycal_{R_n}}(K_n,\Ocal_{\Delta_n})) \\
      &\simeq \varprojlim \Hom_{\Ycal_{R_n}}(K_n,\Ocal_{\Delta_n}),
    \end{aligned}
  \end{equation}
  where we have used that limits commute with taking global sections.
  Therefore the sequence $(s_n)_n$ constructed above yields
  a morphism $s$ as in \eqref{equation:morphism-over-R}.

  Let $K_{\cA_R}$ be the cone of $s$, so we get a distinguished triangle
  \begin{equation}
    \label{equation:exact-triangle-deformed}
    K_{\cB_R} \xrightarrow{s} \Ocal_{\Delta_R} \to K_{\cA_R} \to  K_{\cB_R} [ 1 ].
  \end{equation}
  By replacing
  \(
    R _{ n }
  \)
  with~$R$ in the previous argument
  (showing that \eqref{equation:n-th deformation} is indeed a decomposition of the diagonal),
  we similarly obtain the vanishing
  \begin{equation}
    \label{equation:RHom-vanishing}
    \RRlHom_{f_R}(\mathscr E_{K_{\cB_R}},\mathscr E_{K_{\cA_R}}) = 0.
  \end{equation}
  Therefore the distinguished triangle \eqref{equation:exact-triangle-deformed}
  represents an $R$-valued point of the functor $\dec_{\Delta_f}$.

  It remains to prove the well-definedness of $\vartheta^{-1}$;
  i.e.,
  we have to show that the morphism \eqref{equation:morphism-over-R}
  is unique up to isomorphisms.
  The uniqueness of the algebraization does not seem to be contained in
  the statement of \cite[Proposition 3.6.1]{MR2177199},
  and it is not clear to the authors whether the proof of loc.~cit.~implies it or not.
  To directly verify the uniqueness in our setting,
  let~$t \colon K \to \cO_{ \Delta_R }$
  be another lift of~$s_0$.
  The vanishing~\eqref{equation:vanishing on the central fiber} and \cref{corollary:lifting-morphisms-of-perfect-complexes} implies that for each~$n \geq 0$ there exists an isomorphism
  \(
    c _{ n }
    \colon
    K_{\cB_R} \vert_{ \cY_{R_n} } \simto K \vert_{ \cY_{R_n} }
  \)
  which satisfies~$s \vert_{ \cY_{R_n} } = t \vert_{ \cY_{R_n} } \circ c_{n}$. Since isomorphisms of decomposition triangles are unique
  (see \cref{remark:uniqueness of the isomorphism}), for each~$n \geq 0$ we have that
  \(
    c_{ n + 1 } \vert_{ \cY_{R_n} } = c_{ n }
  \).
  Thus we have obtained a compatible system of isomorphisms
  \begin{equation}
    \left(c_{n} \colon K_{\cB_R} \vert_{ \cY_{R_n} }
      \simto
    K \vert_{ \cY_{R_n} }\right)_{ n \geq 0}
    \in
    \varprojlim \Hom_{\Ycal_{R_n}}
    \left(
      K_{\cB_R} \vert_{ \cY_{R_n} }, K \vert_{ \cY_{R_n} }
    \right).
  \end{equation}
  Again by the comparison theorem we obtain the isomorphism
  \begin{equation}
    \begin{tikzcd}
      \Hom_{ \cY_R } ( K_{\cB_R}, K )
      \arrow{r}{\sim} &
      \varprojlim \Hom_{\Ycal_{R_n}} (K_{\cB_R} \vert_{ \cY_{R_n} }, K \vert_{ \cY_{R_n} }),
    \end{tikzcd}
  \end{equation}
  from which we obtain the morphism~$c \colon K_{\cB_R} \to K$
  corresponding to~$( c_{n} )_{ n \geq 0}$.
  We can similarly confirm~$c^{ - 1 } = \varprojlim c_{n}^{ - 1 }$
  and~$s = t \circ c$, concluding the proof.
\end{proof}

\bibliographystyle{amsplain}
\providecommand{\bysame}{\leavevmode\hbox to3em{\hrulefill}\thinspace}
\providecommand{\MR}{\relax\ifhmode\unskip\space\fi MR }
\providecommand{\MRhref}[2]{%
  \href{http://www.ams.org/mathscinet-getitem?mr=#1}{#2}
}
\providecommand{\href}[2]{#2}

\medskip
\emph{Pieter Belmans}, \url{p.belmans@uu.nl} \\
\textsc{Mathematical Institute, Utrecht University, Budapestlaan 6, 3584 CD Utrecht, The Netherlands} \\

\emph{Wendy Lowen}, \texttt{wendy.lowen@uantwerpen.be} \\
\textsc{Department of Mathematics, University of Antwerp, Middelheimlaan 1, 2020 Antwerpen, Belgium} \\

\emph{Shinnosuke Okawa}, \texttt{okawa@math.sci.osaka-u.ac.jp} \\
\textsc{Department of Mathematics, Graduate School of Science, Osaka University, Machikaneyama 1-1,
Toyonaka, Osaka 560-0043, Japan} \\

\emph{Andrea T.~Ricolfi}, \texttt{aricolfi@sissa.it} \\
\textsc{Scuola Internazionale Superiore di Studi Avanzati (SISSA), Via Bonomea 265, 34136 Trieste, Italy}

\end{document}